\documentclass[11pt, a4paper]{amsart}
\usepackage[english]{babel}
\usepackage[utf8]{inputenc}
\usepackage{amsthm, amsmath,amsthm, amssymb}
\usepackage{amsfonts}
\usepackage{color}

\newtheorem{thm}{Theorem}
\newtheorem*{thm*}{Theorem}
\newtheorem{prop}{Proposition}[section]
\newtheorem{lem}[prop]{Lemma}
\newtheorem{ques}{Question}
\newtheorem{cor}[prop]{Corollary}

\theoremstyle{definition}

\newtheorem{rem}[prop]{Remark}
\numberwithin{equation}{section}

\usepackage{varioref}

\title[Integrability of dominated decompositions]
{Integrability of  dominated decompositions on three-dimensional manifolds}

\author{Stefano Luzzatto, S\.ina T\"urel\.i,  Khadim War}

\date{27 March 2015}

\thanks{This paper has benefited from discussions and email exchanges with several people to whom we are very grateful. We mention in
particular Christian Bonatti, Andy Hammerlindl, Rafael Potrie, Jana Rodriguez-Hertz, Ali Tahzibi, Andrew Torok,  Raul Ures, Marcelo Viana,
Amy
Wilkinson. We would also like to thank the referee for useful suggestions regarding the presentation and some extensions of our results, in particular the statement about robustly transitive diffeomorphisms in Theorem \ref{thm:volpres}. }

\begin{document}

\begin{abstract}
We investigate the integrability of 2-dimensional invariant distributions (tangent sub-bundles) which arise naturally in the context of dynamical systems on 3-manifolds.  In particular we  prove unique integrability of dynamically dominated and volume dominated Lipschitz continuous invariant decompositions as well as distributions with some other regularity conditions.
\end{abstract}

\maketitle

\section{Introduction and Statement of Results}

 Let \(  M  \) be a smooth manifold and \(  E \subset TM  \) a distribution of tangent hyperplanes. A basic question concerns the (unique)
 integrability of the distribution \(  E  \), i.e. the existence at every point of a (unique) local embedded submanifold everywhere tangent
 to
 E. For one-dimensional distributions it follows from classical results on the existence and uniqueness of solutions of ODE's that regularity
 conditions suffice: existence is always guaranteed for continuous distributions and uniqueness for Lipschitz continuous distributions. For
 higher dimensional distributions the situation is more complicated and regularity conditions alone cannot guarantee integrability, indeed
 there exists arbitrarily smooth distributions which are not integrable \cite{Le}. It turns out however, that if the distributions are
 realized as \(  D\varphi  \)-invariant distributions for some diffeomorphisms \(  \varphi  \) then some conditions can be formulated which
 imply integrability. More precisely,
 let \( M \) be a Riemannian 3-manifold, \( \varphi: M \to M  \)
a \( C^2 \) diffeomorphism and \(  E \oplus F  \) a continuous \(  D\varphi  \)-invariant tangent bundle decomposition. For definiteness we
shall always assume, without loss of generality, that \(  dim (E) = 2  \) and \(  dim (F)=1  \). We state our results in the following
subsections.

\subsection{Dynamical domination and robust transitivity} 
A diffeomorphism \(  \varphi  \) is \emph{transitive} it there exists a dense orbit, and \emph{robustly transitive} if any \(  C^{1}  \) sufficiently close diffeomorphism is also transitive. 
A \(  D\varphi  \)-invariant decomposition \(  E \oplus F  \) is \emph{dynamically dominated} if there exists a Riemannian metric such that
\begin{equation}\label{dyndom}
\frac{\|D\varphi_{x} |_{E_{x}}\|}{\|D\varphi_{x}|_{F_{x}}\|}  < 1
\end{equation}
for all \(  x\in M  \). Then we have the following result.

\begin{thm}\label{thm:volpres}
Let \(  M  \) be a Riemannian 3-manifold, \( \varphi: M \to M\) a volume-preserving or robustly transitive \(  C^{2}  \) diffeomorphism and \(  E \oplus F  \) a \(
D
\varphi\)-invariant, Lipschitz,  dynamically dominated decomposition. Then E is uniquely integrable.
\end{thm}

\begin{rem} Our dynamical domination condition is usually referred to in the literature simply as \emph{domination}, we use this non-standard
terminology to avoid confusion in view of the fact that we will introduce below another form of domination.
We remark also that the  dynamical domination condition is usually formulated with the co-norm \(  m(D\varphi_{x}|_{F_{x}}):=\min_{v\in F,
v\neq 0}\|D\varphi_{x}(v)\|/\|v\|  \) instead of \(  \|D\varphi_{x}|_{F_{x}}\|  \) but of course the two definitions are equivalent when \(
F
\) is one-dimensional, as here. We will also occasionally call $E$, the dominated bundle. 
\end{rem}

\begin{rem}
We mention that robustly transitive diffeomorphisms always admit a dynamically dominated decomposition \cite{DiPuUr}, so the main assumption in this case is that the decomposition is Lipschitz. 
\end{rem}

\subsection{Volume domination}
We will obtain Theorem \ref{thm:volpres} as a special case of the following more general result which replaces the volume preservation and the robust transitivity condition with a volume ``domination'' condition.
A \(  D\varphi  \)-invariant decomposition \(  E \oplus F  \) is \emph{volume dominated} if there exists a Riemannian metric such that 
\begin{equation}\label{voldom}
\frac{|det(D\varphi_{x} |_{E_{x}})|}{|det(D\varphi_{x}|_{F_{x}})|}  < 1
\end{equation}
for all \(  x\in M  \).  Then we have the following result.
 \begin{thm}\label{thm:dom}
Let \(  M  \) be a Riemannian 3-manifold,  \( \varphi: M \to M \) a \( C^2 \) diffeomorphism and \(  E \oplus F  \) a \(  D\varphi
\)-invariant, Lipschitz continuous, dynamically and volume dominated, decomposition. Then \( E \) is uniquely integrable.
 \end{thm}

Theorems \ref{thm:volpres} and  \ref{thm:dom} extend analogous statements in
\cite{LuTuWa} obtained using different arguments, in arbitrary dimension but under the  assumption that the decomposition is \(
C^{1}  \).
They also extend previous results of Burns and Wilkinson \cite{BurWil08}, Hammerlindl and Hertz-Hertz-Ures \cite{Ha, HHU07} and Parwani
\cite{Par} who prove analogous results\footnote{In some of the references mentioned, the relevant results are not always stated in the same form as given here but may be derived from related statements and the technical arguments. In some cases the setting considered is that of partially
hyperbolic diffeomorphisms with a tangent bundle decomposition of the form \(  E^{s}\oplus E^{c}\oplus E^{u}  \) where \(  E^{s}  \) is
uniformly contracting and \(  E^{u}  \) uniformly expanding. In this setting, one considers the integrability of the sub-bundles \(
E^{sc}=E^{s}\oplus E^{c}  \) and \(  E^{cu}=E^{c}\oplus E^{u}  \) and it is not always completely clear to what extent the existence of a
uniformly expanding sub-bundle is relevant to the arguments. We emphasize that the setting we consider here does not require the invariant
distribution \(  E  \) to contain any further invariant sub-bundle.
}
for respectively \(  C^{2}, C^{1}  \) and Lipschitz distributions under the assumption of
\emph{center-bunching} or \emph{2-partial hyperbolicity}:
 \begin{equation}\label{eq-bunched}
 \frac{\|D\varphi_{x}|_{E_{x}}\|^2}{\|D\varphi_{x}|_{F_{x}}\|}<1
 \end{equation}
for every \(  x\in M  \). In the 3-dimensional setting condition \eqref{eq-bunched}  clearly implies volume
domination and is therefore  more restrictive.  In Section \ref{dom-hyp} we sketch an example of a diffeomorphism and an invariant
distribution \(  E  \) which does not satisfy condition \eqref{eq-bunched} but does satisfy the dynamical domination and volume domination
assumptions we require in Theorem \ref{thm:dom}. This particular example is uniquely integrable by construction and so is not a ``new''
example, but helps to justify the observation that our conditions are indeed less restrictive than center-bunching \eqref{eq-bunched}.

The techniques we employ here are similar to those of Parwani but to relax the center-bunching condition one needs a more careful analysis of
the behaviour of certain Lie brackets, this  is carried out in Section \ref{Lie}.

A more sophisticated version of our arguments also yields
 an alternative sufficient condition for integrability which is related to
bundles which are Lipschitz along a transversal direction. Since the conditions are somewhat technical and it is not completely clear if they  are satisfied by any natural examples we have ``relegated'' the precise formulation and proof to the Appendix. We do think nevertheless that both the statement and the techniques used  are of some independent interest and  discuss this further below.

\begin{rem} The assumption that the diffeomorphisms in the Theorems above are \( C^{2} \)
is  necessary for the arguments we use in the proofs.
In the proof of Theorem \ref{thm:dom}, we need  to be able to compute the Lie brackets of iterates of certain sections from $E$ by $D\varphi$. For this reason $D\varphi$ needs to be $C^1$ to keep the regularity of a section along the orbit of a initial point $p$.
\end{rem}

 \subsection{Volume Domination versus 2 Partially Hyperbolic}\label{dom-hyp}
In this section we are going to sketch the construction of some non-trivial examples which satisfy the volume domination condition
\eqref{voldom} but not the center-bunching condition \eqref{eq-bunched}. This is a variation of the ``derived from Anosov'' construction due
to Ma\~n\'e \cite{Ma} (see \cite{VB} for the volume preserving case, which is what we use here). We are very grateful to Ra\'ul Ures for
suggesting and explaining this construction to us.
Consider the  matrix
$$
\left(
  \begin{array}{ccc}
    -3 & 0 & 2 \\
    1 & 2 & -3 \\
    0 & -1 & 1 \\
  \end{array}
\right)
$$
This matrix has determinant $1$ and has integer coefficients therefore induces a volume preserving toral automorphism on $\mathbb{T}^3$. It
is
Anosov since its eigenvalues are $r_1\sim -0.11, r_2\sim 3.11, r_3\sim-3.21$. Note that $r_1r_2/{r_3}<1$ but ${r_2^{2}}/{r_3}>1$. Hence
\eqref{voldom} is satisfied but \eqref{eq-bunched} is not.
Now  take a fixed point $p$ and a periodic point $q$ and a neighbourhood $U$ of $p$ so that forward iterates of $q$ never intersect $U$. One
can apply Man\'e's construction to perturb the map on $U$ as to obtain a new partially hyperbolic automorphism of $\mathbb{T}^3$ which is
still volume preserving. Such a perturbation is not a small one and therefore one can not claim integrability of the new system trivially by
using standard theorems as in \cite{HiPuSh}. Since the perturbation is performed on $U$ and orbit of $q$ never intersects $U$, the
perturbation does not change the  splitting and the contraction and expansion rates around $q$ and in particular \eqref{eq-bunched} is still
not satisfied on the orbit of $q$. Yet the new example is volume preserving therefore it is necessarily the case that \eqref{voldom} is
satisfied.

\section{Strategy and overview of the proof}\label{sec-overview}
We will first show that Theorem \ref{thm:volpres} is a special case of Theorem  \ref{thm:dom}. We will consider the volume preserving setting and the robustly transitive setting separately. We then discuss the proof of Theorem \ref{thm:dom}. 

\subsection{Volume preserving implies volume domination}
We show that when 
\(  \varphi  \) is volume preserving,   dynamical domination implies volume domination. Indeed, notice that
 \(  |det(D\varphi_{x}|_{F_{x}})| =  \|D\varphi_{x}|_{F_{x}}\|  \) since \(  F  \) is one-dimensional, so the difference
between dynamical domination and volume domination consists of the difference  between \(  \|D\varphi_{x} |_{E_{x}}\|  \) and \(
|det(D\varphi_{x} |_{E_{x}})|  \). These two quantities are in general essentially independent of each other; indeed considering the singular
value decomposition of \(  D\varphi_{x}|_{E_{x}}  \) and letting \(  s_{1}\leq s_{2}  \) denote the two singular values (since we assume \(
E
\) is 2-dimensional), we have that \(   \|D\varphi_{x} |_{E_{x}}\| = s_{2} \) and \(  |det(D\varphi_{x} |_{E_{x}})| = s_{1}s_{2} \). If \(
\|D\varphi_{x} |_{E_{x}}\| = s_{2}< 1 \) then we have a a straightforward inequality \(  |det(D\varphi_{x} |_{E_{x}})| = s_{1}s_{2} < s_{2} =
\|D\varphi_{x} |_{E_{x}}\| \) but this is of course not necessarily the case in general. However there is a relation in the   volume preserving  setting as this
implies
\( |det{D\varphi_x|_E}|\cdot |det D\varphi_x|_F|=1\) and so \eqref{dyndom} implies \(  |det D\varphi_x|_F|> 1  \) (arguing by contradiction,
\(  |det D\varphi_x|_F|=\|D\varphi_x|_F\| \leq 1  \) would imply  \(  |det D\varphi_{x}|_{E}|\geq 1  \) by the volume preservation, and this
would imply \(  \|D\varphi_{x}|_{E}\|/ \|D\varphi_{x}|_{F}\|\geq 1   \) which  would contradict \eqref{dyndom}). Dividing the equation  \(
|det{D\varphi_x|_E}|\cdot |det D\varphi_x|_F|\) \( =1\)  through by \(  (|det D\varphi_x|_F)^{2}|  \) we get \eqref{voldom}.

\subsection{Robust transitivity implies volume domination}
 It is shown in \cite{DiPuUr} that  if $\varphi$ is a $C^1$ robustly transitive diffeomorphism of a compact 3-manifold then it admits a dominated splitting such that at least on of the following two conditions hold: \emph{i)} \(  \varphi  \) admits a splitting of the form $E^s \oplus E^c \oplus E^u$ and there exists a Riemannian metric  
 such that \(  \|D\varphi|_{E^{s}_{x}} \|  <  \|D\varphi|_{E^{c}_{x}} \| <  \|D\varphi|_{E^{u}_{x}} \|  \) and \(   \|D\varphi|_{E^{s}_{x}} \| <1< \|D\varphi|_{E^{u}_{x}} \|   \); \emph{ii)}
either $\varphi$ or $\varphi^{-1}$ admits a splitting of the form $E \oplus F$ for which $dim(E)=2$ and \(  E  \) is volume contracting, i.e.  there exists a Riemannian metric such that 
$
|det(D\varphi|_{E_{x}})| < 1< \|D\varphi |_{F_{x}}\|
$
for all $x \in M$. Both of these conditions easily imply volume domination for $E$.

\subsection{Volume domination implies integrability}
From now on we concentrate on Theorem \ref{thm:dom} and reduce it to a key technical Proposition.
We fix  an arbitrary point \(  x_{0}\in M  \) and a local chart $(\mathcal U,x^1,x^2,x^3)$ centered at $x_0$.
We can assume (up to change of coordinates) that
${\partial}/{\partial x^i},i=1,2,3$ are transverse to $E$ and thus we can define linearly independent vector fields  ${X}$ and ${Y}$, which span $E$ and are of the form
$$
X=\frac{\partial}{\partial x^1}+a\frac{\partial}{\partial x^{3}}
 \quad \quad \quad
 Y=\frac{\partial}{\partial x^2}+b\frac{\partial}{\partial x^{3}}.
 $$
where $a$ and $b$ are Lipschitz functions. Notice that it follows from the form of the vector fields \(  X, Y  \) that at every point of differentiability the Lie bracket is well defined and lies in the \(  x^{3}  \) direction, i.e. 
\[
[X, Y] = c \frac{\partial}{\partial x^{3}}
\]
for some \(  L^{\infty}  \) function \(  c  \). 
In Section \ref{Lie} we will prove the following
\begin{prop}\label{Prop:as}
 There exists $C>0$ such that for every $k>1$ and $x\in\mathcal{U}$,
 if the distribution $E$ is differentiable at $x$ then we have
 $$\|[X,Y]_{x}\|\leq C 
 \frac{|det(D\varphi^k_x|_{E_x})|}{|det (D\varphi^k_x|_{F_x})|}.$$
\end{prop}
Substituting the volume domination condition \eqref{voldom} into the estimate in Proposition \ref{Prop:as}, we get that the right hand side converges to 0 as \( k\to \infty \), and therefore
\( \|[X,Y]_{x}\| = 0 \) and so  the distribution \( E \) is involutive at every point \( x \) at which it is differentiable. Theorem \ref{thm:dom} is then an immediate consequence of the
following  general result of  Simi\'c \cite{S} which holds in arbitrary dimension and  a generalization of a well-known classical result of Frobenius proving unique integrability for involutive \( C^{1} \) distributions.

\begin{thm*}[\cite{S}]\label{thm-Simic} Let $E$ be an $m$ dimensional Lipschitz distribution on a smooth manifold $M$. If for every point $x_0\in M$,
there exists a local neighbourhood $\mathcal U$ and a local Lipschitz frame $\{X_i\}_{i=1}^m$ of $E$ in $\mathcal U$ such that for almost
every point $x \in\mathcal U$,
$[X_i,X_j]_{x} \in E_{x}$, then $E$ is uniquely integrable.
\end{thm*}

\begin{rem}
We mention that there are some versions of Proposition \ref{Prop:as} in the literature for \( C^{1} \) distributions and giving an estimate of the the
$\|[X,Y]_{x}\|\leq  {|D\varphi^k_x|_{E_x}|^2}/{m(D\varphi_x|_{F_x})}$, , see e.g. \cite{HHU07, Par},.  For this quantity to go to zero, one
needs
the center bunching assumption \eqref{eq-bunched}. In our proposition, through more careful analysis, we relax the condition of center bunching
\eqref{eq-bunched} to
volume domination \eqref{voldom}.
\end{rem}


\section{Lie bracket bounds}\label{Lie}

This section is devoted to the proof of Proposition \ref{Prop:as}, which is now the only missing component in the proof of Theorems
\ref{thm:volpres} and \ref{thm:dom}.
As a first step in the proof, we reduce the problem to that of estimating the norm
of a certain  projection of the  bracket of an orthonormal frame. More specifically, let   \(  \pi  \) denote the orthogonal projection (with
respect to the
Lyapunov metric which orthogonalizes the bundles $E$ and $F$) onto \(  F \).

\begin{lem}\label{proj}
 There exists a constant \(  C_{1}>0  \) such that if
  \( \{Z, W\}\) is an orthonormal Lipschitz
  frame  for \(  E  \) and differentiable at $x\in\mathcal{U}$
 then we have
\[
\| [X, Y]_x\| \leq C_{1} \|\pi[Z,W]_x\|.
\]
\end{lem}

\begin{proof}
Notice that since $F$ and $\frac{\partial}{\partial x_3}$ are transverse to $E$, then one has that $K_1\leq ||\pi\frac{\partial}{\partial
x^3}||\leq K_2$ for some constants $K_1,K_2>0$. Moreover since
$\|\pi[X,Y]\|=|c|.\|\pi{\partial}/{\partial x^3}\|$ and $\|[X,Y]\|=  |c|$ then it is sufficient to get an upper bound for \(  \|\pi[X,Y]\|
\).
Writing \(  X, Y  \) in the local orthonormal frame \(  \{Z, W\}  \) we have
  $$
  X=\alpha_1Z+\alpha_2W \quad \text{ and } \quad
  Y=\beta_1Z+\beta_2W.
  $$
By bilinearity of the Lie bracket
and the fact that \(  \pi(Z)= \pi(W)=0  \) since \(  \pi  \) is a projection along \(  E  \),
  straightforward calculation gives
  \[
  \|\pi[X,Y]\|= |\alpha_1\beta_2-\alpha_2\beta_1|.\|\pi[Z,W]\|
  \]
By orthonormality of \( \{Z, W\}\),  we have   $|\alpha_i|\leq \|X\|,|\beta_i|\leq\|Y\|$ and since these are uniformly bounded, the same is
true for
  $ |\alpha_1\beta_2-\alpha_2\beta_1|$ and so we get the result.
 \end{proof}

By Lemma \ref{proj} it is sufficient to obtain an upper bound for the quantity
\(  \|\pi[Z,W]\|  \) for some Lipschitz orthonormal frame.
In particular we can (and do) choose Lipschitz orthonormal frames \(  \{Z, W\}  \) of \(  E  \)
 such that for every \(  x\in \mathcal U  \) and every \(  k\geq 1 \) we have
  $$
  \|D\varphi^k_x Z\|\|D\varphi^k_xW\|=|det(D\varphi^k|_{E})|.
  $$
  For these frames will we prove the following.

 \begin{lem}\label{Prop:Key1}
There exists $C_{2}>0$ such that for every \( k\geq 1\)  and $x\in\mathcal{U}$, if the distribution $E$ is differentiable
at $x$  we have
$$
\|\pi[Z,W]_{x}\|\leq C_{2} \frac{|det(D\varphi^k|_{E_x})|}{ ||D\varphi^k|_{F_x}||}.
$$
\end{lem}
Combining Lemma \ref{Prop:Key1} and Lemma \ref{proj} and letting \( C=C_{1}C_{2} \) we get:
$$
\|[X,Y]_x\| \leq C_{1}\|\pi[Z,W]_x\| \leq C_{1}C_{2} \frac{|det(D\varphi^k_x|_{E_{x}})|}{\|D\varphi^k_x|_{F_{x}}\|}= C \frac{|det(D\varphi^k_x|_{E_{x}})|}{|det(D\varphi^k_x|_{F_{x}})|}
$$
which is the desired bound in Proposition \ref{Prop:as} and therefore completes its proof.

To prove Lemma \ref{Prop:Key1},  observe first that for every $y\in M$  there exist $2$ orthonormal  Lipschitz  vector fields $A_y,B_y$ that span $E$
 in a neighborhood of $y$ and by compactness we can suppose that we have finitely many pairs, say
 $(A_1,B_1),...,(A_\ell,B_\ell)$ of such vector fields which together cover the whole manifold.
  We denote by $\mathcal{U}_i$ the domain
 where the vector fields $A_i,B_i$ are defined and let
 \[
 C_{2}:=\sup\{|\pi[A_i,B_i](x)|: 1\leq i\leq l \text{ and almost every }  x\in\mathcal U_i\}.
 \]
Note this constant $ C_{2}$ is finite. In fact, by the standard fact that Lipschitz functions have weak differential which is essentially bounded (
or $L^{\infty}$ ), then for every $i\in\{1,...,l\}$
the function $|[A_i,B_i]|$ is bounded. To complete the proof we will use the following observation.

\begin{lem}\label{lem:equal}
For any Lipschitz orthonormal local frame
$\{Z,W\}$   for $E$  which is differentiable at $x\in M$, we have
 $$|\pi[Z,W]|\leq  C_{2}$$
\end{lem}
\begin{proof}
 Write
 \(
 Z=\alpha_1A_{i}+\alpha_2B_{i} \)
 and
 \(
 W=\beta_1A_{i}+\beta_2 B_{i}
 \)
 for some \(  1\leq i \leq \ell  \).
 Using the bilinearity of the Lie bracket and the fact that $\pi(A_{i})=\pi(B_{i})=0$ we get
 $
 |\pi[Z,W]|=|\alpha_1\beta_2-\alpha_2\beta_1|| \pi[A_{i},B_{i}]|.
 $
 Since \(  \{A_{i}, B_{i}\}  \) and $\{Z,W\}$ are both orthonormal frames,
we have $|\alpha_1\beta_2-\alpha_2\beta_1|=1$, and so we get result.
\end{proof}

\begin{proof}[Proof of lemma \ref{Prop:Key1}]
For $k>k_0$ and $x\in\mathcal{U}$ such that $E$ is differentiable at $x$, Let
  \[
  \tilde Z(\varphi^kx) = \frac{D\varphi^k_x Z}{\|D\varphi^k_xZ\|}
\quad \text{ and } \quad
  \tilde W(\varphi^kx) = \frac{D\varphi^k_xW}{\|D\varphi^k_x  W\|}
  \]
Recall that \( D\varphi^k_{x} (E_{x}) = E_{\varphi^{k}(x)}  \). Therefore, since
 $Z, W$  span $E$  in a neighborhood
 of $x$, then  $\tilde Z, \tilde W$ span \(  E  \) in
 a neighbourhood of \(  \varphi^{k}(x)  \) and in particular $\pi(\tilde Z)=\pi(\tilde W)=0$.
Therefore we get
 \begin{equation}\label{eq1}
 \|\pi[D\varphi^k Z,D\varphi^k W]\|=|det(D\varphi^k|_{E^{(k)}})|\|\pi[\tilde Z,\tilde W]\|
 \end{equation}
Note that $\|\pi[D\varphi^k Z,D\varphi^k W]\| = \|\pi D\varphi^k[Z,W]\|$. Then by the invariance of the bundles we have
 \begin{equation}\label{eq2}
  \|\pi D\varphi^k[Z, W]\|=\|D\varphi^k\pi[ Z, W]\|.
  \end{equation}
 Since $F$ is one dimensional,
 \begin{equation}\label{eq3}
 \|D\varphi^k\pi[ Z, W]\|= \|D\varphi^k|_{F}\|\|\pi[ Z, W]\|
 \end{equation}
Combining  \eqref{eq2} and \eqref{eq3} we get
$$
\|D\varphi^k|_{F}\|\|\pi[Z,W]\|=\|\pi D\varphi^k[ Z, W]\|.
$$
Putting this into equation \eqref{eq1} and using the fact that $\|\pi[\tilde Z,\tilde W]\|$ is uniformly bounded by lemma \ref{lem:equal} one
gets
$$
\|\pi[Z,W]\| \leq C_{2} \frac{|det(D\varphi^k|_{E})|}{\|D\varphi^k|_{F}\|}
$$
This concludes the proof of Lemma \ref{Prop:Key1}.
\end{proof}

\appendix

\section{Sequential transversal regularity}

A  counter example in \cite{HeHeUr2} shows that the Lipschitz regularity condition in our Theorems cannot be fully relaxed, without
additional
assumptions, in order to guarantee unique integrability. Nevertheless the kind of techniques we use lead naturally to the formulation of a
somewhat
unorthodox regularity condition, which we call ``sequentially transversal Lipschitz regularity''. The main reason that we choose to present this result is that the techniques used in the proof, especially those in Section \ref{sec-novel}, generalize naturally to yield continuous Frobenius-type theorems, such as those given in forthcoming papers \cite{SSK2, SK}; we also believe that there are some interesting questions to be pursued regarding the relation between Lipschitz regularity and transversal Lipschitz regularity, we discuss these in Section \ref{translip} below.  We will give a ``detailed sketch'' of the arguments concentrating mostly on techniques which are novel,  the full arguments can be found in a previous version of this paper \cite{LuzTurWar2}.

\subsection{Definition and statement of result}
As above, let  \(  M  \) be a 3-manifold,  \( \varphi: M \to M \) be a \( C^2 \) diffeomorphism, and  \( E\oplus F \) a continuous \(
D\varphi  \)-invariant tangent bundle decomposition with \( dim (E)=2\).
 We say that \( E \) is
 \textit{sequentially transversally Lipschitz} if there exists a $C^1$ line bundle $Z$, everywhere transverse to $E$,  and a \( C^1 \)
 distribution \( E^{(0)} \) such that the
sequence of \( C^1 \) distributions \(  \{E^{(k)}\}_{k>1}  \) given by
\begin{equation}\label{pullback}
E^{(k)}_x=D\varphi^{-k}_{\varphi^kx}  E^{(0)}_{\varphi^kx}, \forall x\in M, k>1
\end{equation}
 are equi-Lipschitz along  $Z$, i.e.  there
 exists $K>0$ such that for every $x,y\in M$ close enough and belonging to the same integral curve of $Z$, and every \(  k\geq 0  \),  we
 have
 $\measuredangle(E^{(k)}_x,E^{(k)}_y)\leq Kd(x,y)$.

 \begin{thm}\label{thm:coh-relax}
Let \(  M  \) be a Riemannian 3-manifold,  \( \varphi: M \to M \) a \( C^2 \) diffeomorphism and \(  E \oplus F  \) a \(  D\varphi
\)-invariant, sequentially transversally Lipschitz, dynamically and volume dominated, decomposition. Then \( E \) is uniquely integrable.
 \end{thm}

\subsection{Relation Between Lipschitzness and Transversal Lipschitzness}
\label{translip}

Before starting the sketch of the proof of Theorem  \ref{thm:coh-relax} we discuss some general questions concerning the relationships between various forms of Lipschitz regularity.
We say that  a sub-bundle E is  \emph{transversally Lipschitz} if  there exists a $C^1$ line bundle $Z$, everywhere transverse to $E$, along which \( E \) is Lipschitz. The relations between Lipschitz, transversally Lipschitz, and sequentially transversally Lipschitz are not clear in general. For example it is easy to see that sequentially transversally Lipschitz implies transversally Lipschitz but we have not been able to show that transversally Lipschitz, or even Lipschitz, implies sequentially transversally Lipschitz.
Nevertheless  certain equivalence may exist under certain forms of dominations for bundles which occur as invariant bundles for diffeomorphisms.  We formulate the following question:
\begin{ques}
Suppose  $E \oplus F$ is a $D\varphi$-invariant decomposition satisfying  \eqref{eq-bunched}.  Then is \(  E  \) transversally Lipschitz if
and only if it is Lipschitz ?
\end{ques}

One reason why we believe this question is interesting is that that transversal Lipschitz regularity is a-priori strictly weaker than full Lipschitz regularity. Thus a positive answer to this question would imply that  transversal Lipschitzness of center-bunched dominated systems becomes in particular, by Theorem \ref{thm:dom}, a criterion for
their unique integrability. More generally, a positive answer to this question  would somehow be saying that  one only needs some domination condition and transversal regularity to prevent $E$ from demonstrating pathological behaviours such as non-integrability or non-Lipschitzness.

The notion of sequential transversal regularity and the result of Theorem~\ref{thm:coh-relax} may play a role in a potential solution to the question above. Indeed, if \( E \) is sequentially transversally Lipschitz and volume dominated, then by Theorem \ref{thm:coh-relax} it is uniquely integrable. Then,  under the additional assumption of centre-bunching,  by arguments derived from theory of normal hyperbolicity (see
\cite{HiPuSh}) it is possible to deduce that $E$ is Lipschitz along its foliation $\mathcal{F}$. We also know that there is a complementary
transversal foliation given by integral curves of $Z$ along which $E$ is sequentially Lipschitz and therefore Lipschitz. This implies that $E$ is Lipschitz.

Thus center-bunching and sequential transverse regularity implies Lipschitz.
The missing link would just be to show that if  $E$ is transversally Lipschitz along a direction then it is also sequentially transversally Lipschitz along that direction. This would yield a positive answer to the question.

\subsection{General philosophy and strategy of proof}\label{Strategy}
Since our distribution is no longer Lipschitz we are not able to apply any existing general involutivity/integrability result, such as that of Simi\'c quoted above\footnote{Some notion of Lie bracket can be formulated in lower regularity, see for example \cite[Proposition 3.1]{BuIv}, but it is not clear to us how to obtain a full unique  integrability result using these ideas.}. Instead we will have to essentially construct the required integral  manifolds more or less explicitly ``by hand''.

The standard approach for this kind of construction is the so-called \emph{graph transform} method, see \cite{HiPuSh}, which takes full advantage of certain hyperbolicity  conditions and consists of "pulling back" a sequence of manifolds and showing that the sequence of pull-backs converges to a geometric object which can be shown to be a unique integral manifold of the distribution.
This method goes back to Hadamard and has been used in many different settings but, generally, cannot be applied
in the partially hyperbolic or dominated decomposition setting where the dynamics is allowed to have a wide range of dynamical behaviour and it is therefore  impossible to apply any graph transform arguments to \( E^{sc} \) under our assumptions.  This is perhaps one of the main reasons why this setting has proved so difficult to deal with.

The strategy we use here can be seen as a combination of the Frobenius/Simi\'c involutivity approach and the Hadamard graph transform method. Rather than approximating the desired integral manifold by a sequence of manifolds we approximate  the continuous distribution \( E \) by a sequence \( \{E^{(k)}\} \) of \( C^1 \) distributions  obtained dynamically by "pulling back" a  suitably chosen initial distribution.
Since these  approximate distributions are \( C^1 \), the Lie brackets of \( C^1 \) vector fields in \( E^{(k)} \) can be defined. \emph{If the \( E^{(k)} \) were involutive}, then each one would admit an integral manifold \( \mathcal E^{(k)} \) and it is fairly easy to see that these converge to an integral manifold of the original distribution \( E \). However this is generally not the case and we need a more sophisticated argument to show that the distributions \( E^{(k)} \) are "asymptotically involutive" in a particular sense
which will be defined formally below. For each \( k \) we will construct an "approximate" local center-stable manifold \( \mathcal W^{(k)} \) which is not an integral manifold of \( E^{(k)} \)
(because the \( E^{(k)} \) are not necessarily involutive) but is "close" to being integral manifolds. Further estimates, using also the asymptotic involutivity of the distributions \( E^{(k)} \), then allow us to show that these manifolds converge to an integral manifold of the  distribution \( E \).
 We will then use a separate argument to obtain uniqueness, taking advantage of a result of Hartman.

\subsection{Almost involutive approximations}\label{sec-novel}
In this section we state and prove a generalization of Proposition \ref{Prop:as} which formalizes the meaning of ``almost'' involutive.
We consider the sequence of \( C^{1} \) distributions $\{E^{(k)}\}_{k>1}$ as
in the definition of sequential transverse regularity in \eqref{pullback}.
We fix a coordinate system $(x^{1},x^{2},x^{3},\mathcal{U})$ so that ${\partial}/{\partial x^i}$ are all transverse to $E$ and therefore to
$E^{(k)}$ for $k$ large enough since $E^{(k)} \rightarrow E$ uniformly in angle. Then thanks to this transversality assumption we can find
vector fields defined on $\mathcal U$ of the form
\begin{equation}\label{eq-vecfieldsk}
X^{(k)}=\frac{\partial}{\partial x^1}+a^{(k)}\frac{\partial}{\partial x^{3}}
 \quad \text{ and } \quad
 Y^{(k)}=\frac{\partial}{\partial x^2}+b^{(k)}\frac{\partial}{\partial x^{3}}.
\end{equation}
that span $E^{(k)}$ and converge to vector fields of the form
\begin{equation}\label{eq-vecfields}
X=\frac{\partial}{\partial x^1}+a\frac{\partial}{\partial x^{3}}
 \quad \text{ and } \quad
 Y=\frac{\partial}{\partial x^2}+b\frac{\partial}{\partial x^{3}}.
 \end{equation}
that span $E$ for $a^{(k)}$, \( a \), $b^{(k)}$, \( b \) everywhere non-vanishing functions. Moreover we can choose ${\partial}/{\partial x^3}$ to be the direction
where
sequential transversal Lipschitzness holds true (since it is already transversal to $E$) so we have the property that there exists $C>0$
\begin{equation}\label{sec}
 \left|\frac{\partial a^{(k)}}{\partial x^3}\right| <C \quad\text{ and }\quad \left|\frac{\partial b^{(k)}}{\partial x^3}\right| <C
\end{equation}
for all $k$. It is easy to check that $[X^{(k)},Y^{(k)}]$ lies in the \( \partial/\partial x^3 \) direction. As before we will have some
estimates about how fast the Lie brackets of these vector fields decay to $0$. For the following let $F^{(k)}$ be the continuous bundle which
is orthogonal to $E^{(k)}$ with respect to Lyapunov metric on $E$ so that $F^{(k)}$ goes to $F$ in angle (since $F$ is orthogonal to $E$ with
respect to the Lyapunov metric). We have the following analogue of Proposition \ref{Prop:as}
\begin{prop}\label{Prop:asap}
 There exists $C>0$ such that for every $k>1$ and $x\in\mathcal{U}$,
 we have
 $$\|[X^{(k)},Y^{(k)}](x)\|\leq C\frac{|det(D\varphi^k_x|_{E^{(k)}_x})|}{\|D\varphi^k_x|_{F^{(k)}_x}\|}.$$
\end{prop}

\begin{proof}[Sketch of proof]  The proof of Proposition \ref{Prop:asap} is very similiar to that of Proposition \ref{Prop:as} and  it is not hard to get the result with the difference that $E$ and $F$ in the right
hand side of the estimate in Proposition \ref{Prop:as} are replaced by $E^{(k)}$ and  $F^{(k)}$.
In this case we choose our collection of $C^1$ orthonormal collection of frames $\{Z^{(k)},W^{(k)}\}$ of $E^{(k)}$ so that
$
||D\varphi^kZ^{(k)}||||D\varphi^kW^{(k)}|| = det(D\varphi^k|_{E^{(k)}})
$
Then exactly as in lemma \ref{lem:equal}, to get an upper bound on $|[X^{(k)},Y^{(k)}]|$, it is enough to bound $[Z^{(k)},W^{(k)}]$. The proof of the inequality
$\|[Z^{(k)},W^{(k)}]\| \leq {det(D\varphi^k|_{E^{(k)}})}/{\|D\varphi^k|_{F^{(k)}}\|}$ follows quite closely the proof of lemma
\ref{Prop:Key1} where
the projection $\pi$ is replaced by $\pi^{(k)}$ which the projection to $F^{(k)}$ along $E^{(k)}$ at relevant places.
\end{proof}

The next, fairly intuitive but in fact quite technical, step is to replace the estimates on the approximations with estimates on the limit bundle.

\begin{prop}\label{Prop:detap}
There exists $C>0$ such that for every $k>1$ and $x \in \mathcal{U}$,
we have
\begin{equation}\label{eq-det}
|det(D\varphi^k_x|_{E^{(k)}_x})| \leq C |det(D\varphi^k_x|_{E_x})|
\end{equation}
and
\begin{equation}\label{eq-F}
\|D\varphi^k_x|_{F^{(k)}_x}\| \geq C \|D\varphi^k_x|_{F_x}\|
\end{equation}
\end{prop}

\begin{proof}[Sketch of proof]
\eqref{eq-F} is fairly easy since for any vector $v\notin E$,
$|D\varphi^kv| \geq  C D\varphi^k_x|_{F_x}|v|$ (since $F$ has dimension $1$). The real technical estimate is \eqref{eq-det}. One first needs
to make the observation that there exists a cone $C(\alpha)$ of angle $\alpha$ around $E$ such that $D\varphi^kE^{(k)} \subset C(\alpha)$.
This is the main observation that allows to relate two determinants independent of $k$. Indeed given a basis $v^k_1, u^k_1$ of $E^{(k)}$,
then
$K|det(D\varphi^k_x|_{E^{(k)}_x})|= |D\varphi^k_xv^k_1 \wedge  D\varphi^k_xu^k_1| = |\Lambda(D\varphi^k_x) v^k_1 \wedge u^k_1|$ where
$\Lambda(D\varphi^k_x)$ is the induced action of $D\varphi^k_x$ on $TM \wedge TM$. But then $\Lambda(D\varphi^k_x)$ allows a dominated
splitting of
$TM \wedge TM$ whose invariant spaces are $E_1 = E \wedge E$, $E_2 = E\wedge F$ where $E_1$ is dominated by $E_2$. We have that
for all $k$, $ E_1^{(k)} = \text{span}(v^k_1 \wedge u^k_1)$ is a space which is inside a cone $C(\alpha)$ around $E_1$. Therefore usual
dominated splitting estimates give that
$
|\Lambda(D\varphi^k_x)|_{E_1^{(k)}(x)}| \leq K |\Lambda(D\varphi^k_x)|_{E_1(x)}|
$
which proves the claim about determinants since $|\Lambda(D\varphi^k_x)|_{E_1(x)}| = |det(D\varphi^k_x|_{E_x})|$.
\end{proof}

\subsection{Almost Integral Manifolds}

We will  use the local frames \(\{X^{(k)}, Y^{(k)}\}  \) to define a family of local manifolds which we will then show converge to the
required integral manifold of \(  E  \). We emphasize that these are \emph{not} in general integral manifolds of the approximating distribution \( E^{(k)} \).
We will use the relatively standard notation  \( e^{tX^{(k)}} \) to denote the flow at time $t\in\mathbb{R}$ of the vector field \( X^{(k)}
\). Then we let
$$
\mathcal W^{(k)}_{x_0}(t,s):=e^{tX^{(k)}}\circ e^{sY^{(k)}}(x_0).
$$
This map is well defined for all sufficiently small \( s,t \) so that the composition of the corresponding flows remains in the local chart \(
\mathcal U \) in which the vector fields \( X^{(k)}, Y^{(k)} \) are defined. Since the vector fields \( X^{(k)}, Y^{(k)}\) are uniformly
bounded in norm,  choosing \( \epsilon \) sufficiently small  the functions \( \mathcal W^{(k)}_{x_0} \) can be defined in a fixed domain
$U_{\epsilon}=(-\epsilon,\epsilon)\times(-\epsilon,\epsilon)$ independent of \( k \) such that
\( \mathcal W^{(k)}_{x_0}(U_\epsilon)\subset \mathcal U.
 \)
 By a direct application of the  chain rule and the definition of \( \mathcal W^{(k)}_{x_0} \),
 for every
 $(t,s)\in U_{\epsilon}$ we have
\begin{equation*}
\tilde{X}^{(k)}(t,s)=\frac{\partial \mathcal W^{(k)}_{x_0}}{\partial t}(t,s)=X^{(k)}(\mathcal W^{(k)}_{x_0}(t,s))
\end{equation*}
and
\begin{equation*}
\tilde{Y}^{(k)}(t,s)=\frac{\partial \mathcal W^{(k)}_{x_0}}{\partial s}(t,s)=(e^{tX^{(k)}})_*Y^{(k)}(\mathcal W^{(k)}_{x_0}(t,s)).
\end{equation*}
where for two vector fields $V, Z$ and $t\in\mathbb{R}$, $(e^{tV})_*Z$ denotes that pushforward of $Z$ by the flow of $V$ defined by
$$
[(e^{tV})_*Z]_p = De^{tV}_{e^{-tV}(p)}Z_{e^{-tV}(p)}.
$$
The following lemma gives a condition for this family of maps to have a convergent subsequence whose limits becomes a surface
 tangent to $E$:

 \begin{lem}\label{lem:tangent}
  If $\tilde{X}^{(k)} \rightarrow X$ and $\tilde{Y}^{(k)} \rightarrow Y$ then the images of $\mathcal W^{(k)}_{x_0}$ are embedded submanifolds and this sequence of submanifolds has  a convergent subsequence whose limit is an integral manifold of $E$.
 \end{lem}

\begin{proof}
Since $X$ and $Y$ are linearly independent by assumption of convergence, the differential of the map $\mathcal{W}^{(k)}_{x_0} $ is invertible
at every point \( (t,s) \in U_\epsilon \), i.e.  the partial derivatives
${\partial \mathcal W^{(k)}_{x_0}}/{\partial s} $ and \( {\partial \mathcal W^{(k)}_{x_0}}/{\partial t} \)
are linearly independent for every \( (t,s) \in U_\epsilon \).
Thus  the maps
\( \mathcal W^{(k)}_{x_0} \) are embeddings and  define
submanifolds through  \( x_0 \) (which are not in general integral manifolds of  \(  E^{(k)}  \)). Moreover, since \(
X^{(k)}, Y^{(k)} \) have uniformly bounded norms, it follows by Proposition \ref{prop:Push} that $D\mathcal W^{(k)}_{x_0}$ has bounded norm
uniformly in $k$ and therefore
the family \( \{\mathcal W^{(k)}_{x_0}\} \) is a compact family in the \( C^1 \) topology.
By the Arzela-Ascoli Theorem this family has a subsequence converging to some limit
 \begin{equation}
\mathcal W_{x_0}: U_\epsilon \to \mathcal U.
\end{equation}
 We claim that \( \mathcal W_{x_0} (U_\epsilon) \) is an integral manifold of \( E \).
Indeed, as \( k\to \infty \),  $X^{(k)} \rightarrow X$, $Y^{(k)} \rightarrow Y$  and  \( \{X, Y\} \) is a local frame
 of continuous vector fields for \( E \), in particular \( X, Y \) are linearly independent and span the distribution \( E \).
Moreover, by Proposition \ref{prop:Push}, the partial derivatives
${\partial \mathcal W^{(k)}_{x_0}}/{\partial t} $ and \( {\partial \mathcal W^{(k)}_{x_0}}/{\partial s} \)
are converging uniformly to $X$ and $Y$ and therefore
$$\frac{\partial \mathcal W_{x_0}}{\partial t}=X
\quad \text{ and }\quad
\frac{\partial \mathcal W_{x_0}}{\partial s}=Y.$$
This shows that \( \mathcal W_{x_0}(U_\epsilon) \) is a $C^1$ submanifold and its tangent space coincides with \( E \) and thus \( \mathcal
W_{x_0}(U_\epsilon) \) is an integral manifold of \( E \), thus proving integrability of \(  E  \) under these assumptions.
\end{proof}

It therefore just remains to verify the assumptions of Lemma \ref{lem:tangent}, i.e.
 to show that the vectors \( X^{(k)} \) and \( (e^{tX^{(k)}})_*Y^{(k)} \) converge to $X$ and $Y$. The first convergence is
obviously true. Thus it remains to show the latter which we show in the next result, thus completing the proof of the existence of integral manifolds.

\begin{prop}\label{prop:Push}
For all  $t\in (-\epsilon, \epsilon)$ we have
 $$\lim_{k\to\infty}\|(e^{tX^{(k)}})_*Y^{(k)}-Y^{(k)}\|=0.$$
\end{prop}

\begin{proof} To obtain this proposition one uses the following standard property of the pushforward (see  proof of Proposition 2.6 in \cite{AgSa} for instance):
\begin{equation}\label{eq-derivative}
\frac{d}{dt}[(e^{tX^{(k)}})_*Y^{(k)}-Y^{(k)}] = (e^{tX^{(k)}})_* [X^{(k)},Y^{(k)}]\end{equation}
together with the following proposition:

\begin{prop}\label{lem:pushbound}
There exists \(  C>0 \) such that for every $k\geq 1, x\in\mathcal{U}$  and  $|t|\leq \epsilon$, we have
$$||(e^{tX^{(k)}})_*\frac{\partial}{\partial x^{3}}|_x\|=\exp{\int_0^t\frac{\partial a^{(k)}}{\partial x^3}\circ
e^{-\tau X^{(k)}}(x)d\tau}$$
\end{prop}
This latter proposition follows by integrating the equality
$$
\frac{d}{dt}((e^{tX^{(k)}})_*\frac{\partial}{\partial x^3}|_x)=\left(e^{tX^{(k)}}\right)_*[X^{(k)},\frac{\partial}{\partial x^3}]|_x
$$
Once this is established since we know that ${\partial a^{(k)}}/{\partial x^3}$ is uniformly bounded \eqref{sec}, we obtain that the
effect of $(e^{tX^{(k)}})_*$ on ${\partial}/{\partial x^{3}}$ is bounded. Since $[X^{(k)},Y^{(k)}]$ is a vector in this direction whose
norm goes to $0$ we directly obtain by equation \eqref{eq-derivative} that $\frac{d}{dt}[(e^{tX^{(k)}})_*Y^{(k)}-Y^{(k)}]$ goes to $0$
uniformly and hence
by mean value theorem that $|(e^{tX^{(k)}})_*Y^{(k)}-Y^{(k)}|$ goes to $0$ which is the proposition.
\end{proof}

\begin{rem}
From the proof, it is seen that the most crucial  ingredient is for the approximations  to satisfy the pushforward bound in
Proposition \ref{prop:Push}. One can generalize this observation to get geometric theorems about integrability of continuous sub-bundles, not
just those arising in dynamical systems, with some additional assumptions such as the Lie brackets going to $0$. This idea, which originated in this paper, is employed in forthcoming works \cite{SSK2, SK}.
\end{rem}

\subsection{Uniqueness}
To get uniqueness of the integral manifolds we will take advantage of a general result of Hartman which we state in a simplified version which is sufficient for our purposes.

\begin{thm}[\cite{H}, Chapter 5, Theorem 8.1]\label{thm-hartmann}
Let $X= \sum_{i=1}^n X^i(t,p)\frac{\partial}{\partial x^{i}}$ be a continuous vector field defined on $I \times U$ where $U\subset \mathbb{R}^n$ and $I \subset \mathbb{R}$.
Let $\eta_i = X^i(t,p)dt - dx^{i}$.  If there exists a sequence of $C^1$ differential forms $\eta^k_i$ such that $|\eta^k_i - \eta_i|_{\infty}
\rightarrow 0$ and $d\eta^k_i$ are uniformly bounded then $X$ is uniquely integrable on $I \times U$. Moreover on compact subset of $U\times I$ the integral curves are uniformly Lipschitz continuous with respect to the initial conditions.

\end{thm}

We recall that a  two form being uniformly bounded is equivalent to each of its component is being uniformly bounded.

\begin{cor}
Vector fields $X$ and $Y$ defined in \eqref{eq-vecfields} are uniquely integrable.
\end{cor}

\begin{proof}
We will give the proof for $X$, that for $Y$ is exactly the same. Since $X$ has the form
$$
X = \frac{\partial}{\partial x^{1}} + a\frac{\partial}{\partial x^3},
$$
its solutions always lie in the $\frac{\partial}{\partial x^{1}},\frac{\partial}{\partial x^{3}}$ plane.
Therefore given a point $(x^1_0,x^2_0,x^3_0)$, it is sufficient to consider the restriction to such a plane.
Then the $C^1$ differential 1-forms defined in Theorem \ref{thm-hartmann} are
$$
\eta_1 = dt - dx^{1} \quad \quad \eta_2 = a(x)dt - dx^{3}
$$
where $x=(x^1,x^2_0,x^3)$, and for the approximations we can write
$$
\eta^{k}_1 = dt - dx^{1} \quad \quad \eta^{k}_2 = a^{(k)}(x)dt - dx^{3}
$$
where $a^{(k)}(x)$ are functions given in equation \eqref{eq-vecfieldsk}, again for some fixed $x^2_0$. But then by sequential transversal Lipschitz assumption and choice of
coordinates we have that $|\frac{\partial a^{(k)}}{\partial x^3}|<C$ for all $k$ and
$$
d\eta^{k}_1 = 0 \quad \quad d\eta^{k}_2 = \frac{\partial a^{(k)}}{\partial x^3}dx^3 \wedge dt
$$
(since we restrict to $x^2=\text{const}$ planes) and the requirements of Theorem \ref{thm-hartmann} are satisfied which proves that $X$ is uniquely integrable.
\end{proof}

Now we have that $X$ and $Y$ are uniquely integrable at every point. Assume there exist a point $p \in\mathcal U$ such that through $p$ there
exist two integral surfaces $\mathcal{W}_1,\mathcal{W}_2$. This means both surfaces are integral manifolds of $E$ and in particular the
restriction of $X$ and $Y$ to their tangent space are uniquely integrable vector fields. Therefore there exists $\epsilon_1$ such
that the integral curve $e^{tX}(p)$ for $|t|\leq \epsilon_1$ belongs to both surfaces. Now consider an integral curve of $Y$ starting at the
points
of $e^{tX}(p)$, that is $e^{sY}\circ e^{tX}(p)$. For $\epsilon_1$ small enough, there exists $\epsilon_2$ small enough such that for every
$|t|< \epsilon_1$ and $|s|< \epsilon_2$ this set is inside both surfaces since $Y$ is also uniquely integrable ($\epsilon_i$ only depend on
norms $|X|$, $|Y|$ and size of $\mathcal{W}_i$ and therefore can be chosen uniformly independent of point). This set is a $C^{1}$ disk  and therefore $\mathcal{W}_1$ and $\mathcal{W}_2$ coincide on
an open domain. Applying this to every point $p\in U$ we obtain that through every point in $U$ there passes a single local integral
manifold.
This concludes the proof of the uniqueness.

\bigskip
Stefano Luzzatto \\
\textsc{Abdus Salam International Centre for Theoretical Physics (ICTP), Strada Costiera 11, Trieste, Italy}\\
 \textit{Email address:} \texttt{luzzatto@ictp.it}

\medskip
Sina T\"ureli \\
\textsc{Abdus Salam International Centre for Theoretical Physics (ICTP), Strada Costiera 11, Trieste, Italy} and \textsc{International School
for Advanced Studies (SISSA), Via Bonomea 265, Trieste}\\
 \textit{Email address:} \texttt{sinatureli@gmail.com}

\medskip
Khadim M. War\\
\textsc{Abdus Salam International Centre for Theoretical Physics (ICTP), Strada Costiera 11, Trieste, Italy} and \textsc{International School
for Advanced Studies (SISSA), Via Bonomea 265, Trieste}\\
 \textit{Email address:} \texttt{kwar@ictp.it}

 \end{document}